\documentclass[11pt]{amsart}
\usepackage{amscd}
\usepackage{amsmath}
\usepackage{amsfonts}
\usepackage{amssymb}
\usepackage{amsthm}
\usepackage{enumerate}
\usepackage[T1]{fontenc}
\usepackage{hyperref}
\usepackage{mathrsfs}
\usepackage{stackrel}
\usepackage[all]{xy}
\usepackage{yhmath}

\DeclareMathAlphabet{\mathpzc}{OT1}{pzc}{m}{it}
\DeclareMathOperator{\gr}{gr}
\DeclareMathOperator{\Hom}{Hom}
\DeclareMathOperator{\End}{End}

\DeclareMathOperator{\Der}{Der}

\DeclareMathOperator{\Sp}{Sp}

\DeclareMathOperator{\Spec}{Spec}

\DeclareMathOperator{\im}{Im}

\DeclareMathOperator{\Ann}{Ann}

\DeclareMathOperator{\op}{op}

\DeclareMathOperator{\mult}{mult}
\DeclareMathOperator{\Char}{char}

\begin{document}
\newtheorem{MainThm}{Theorem}
\renewcommand{\theMainThm}{\Alph{MainThm}}

\theoremstyle{definition}
\newtheorem{defn}{Definition}[section]
\newtheorem{thm}[defn]{Theorem}
\newtheorem{cor}[defn]{Corollary}
\newtheorem{prop}[defn]{Proposition}
\newtheorem{claim}[defn]{Claim}
\newtheorem{lem}[defn]{Lemma}
\newtheorem{ex}[defn]{Example}
\newtheorem{exs}[defn]{Examples}
\newtheorem{quest}[defn]{Question}
\newtheorem{rmk}[defn]{Remark}
\newtheorem{rmks}[defn]{Remarks}
\newtheorem{constr}[defn]{Construction}
\newtheorem{up}[defn]{Universal Property}
\newtheorem{setup}[defn]{Set-up}
\newtheorem{notn}[defn]{Notation}
\newcommand{\Zp}{{\mathbb{Z}_p}}
\newcommand{\Qp}{{\mathbb{Q}_p}}
\newcommand{\Fp}{{\mathbb{F}_p}}
\newcommand{\A}{\mathcal{A}}
\newcommand{\B}{\mathcal{B}}
\newcommand{\C}{\mathcal{C}}
\newcommand{\D}{\mathcal{D}}
\newcommand{\E}{\mathcal{E}}
\newcommand{\F}{\mathcal{F}}
\newcommand{\G}{\mathcal{G}}
\newcommand{\I}{\mathcal{I}}
\newcommand{\J}{\mathcal{J}}
\renewcommand{\L}{\mathcal{L}}
\newcommand{\sL}{\mathscr{L}}
\newcommand{\M}{\mathcal{M}}
\newcommand{\sM}{\mathscr{M}}
\newcommand{\N}{\mathcal{N}}
\newcommand{\sN}{\mathscr{N}}
\renewcommand{\O}{\mathcal{O}}
\newcommand{\cP}{\mathcal{P}}
\newcommand{\R}{\mathcal{R}}
\newcommand{\cS}{\mathcal{S}}
\newcommand{\T}{\mathcal{T}}
\newcommand{\U}{\mathcal{U}}
\newcommand{\sU}{\mathscr{U}}
\newcommand{\sV}{\mathscr{V}}
\newcommand{\V}{\mathcal{V}}
\newcommand{\W}{\mathcal{W}}
\newcommand{\sW}{\mathscr{W}}
\newcommand{\X}{\mathcal{X}}
\newcommand{\Y}{\mathcal{Y}}
\newcommand{\Z}{\mathcal{Z}}
\newcommand{\PFS}{\mathbf{PreFS}}
\newcommand{\h}[1]{\widehat{#1}}
\newcommand{\hA}{\h{A}}
\newcommand{\hK}[1]{\h{#1_K}}
\newcommand{\hsULK}{\hK{\sU(\L)}}
\newcommand{\hsUFK}{\hK{\sU(\F)}}
\newcommand{\hsULnK}{\hK{\sU(\pi^n\L)}}
\newcommand{\hn}[1]{\h{#1_n}}
\newcommand{\hnK}[1]{\h{#1_{n,K}}}
\newcommand{\w}[1]{\wideparen{#1}}
\newcommand{\wK}[1]{\wideparen{#1_K}}
\newcommand{\invlim}{\varprojlim}
\newcommand{\dirlim}{\varinjlim}
\newcommand{\fr}[1]{\mathfrak{{#1}}}
\newcommand{\LRU}[1]{\sU(\mathscr{#1})}
\newcommand{\et}{\acute et}
\newcommand{\ttBan}{{\text{\bfseries\sf{Ban}}}}
\newcommand{\ts}[1]{\texorpdfstring{$#1$}{}}
\newcommand{\st}{\mid}
\newcommand{\be}{\begin{enumerate}[{(}a{)}]}
\newcommand{\ee}{\end{enumerate}}
\newcommand{\qmb}[1]{\quad\mbox{#1}\quad}
\let\le=\leqslant  \let\leq=\leqslant
\let\ge=\geqslant  \let\geq=\geqslant

\title{Bounded linear endomorphisms of rigid analytic functions}
\author{Konstantin Ardakov}
\address{Konstantin Ardakov \\ Mathematical Institute\\University of Oxford\\Oxford OX2 6GG, UK}
\author{Oren Ben-Bassat}
\address{Oren Ben-Bassat \\ Department of Mathematics \\ University of Haifa \\ Haifa 31905, Israel}
\thanks{}
\begin{abstract} Let $K$ be a field of characteristic zero complete with respect to a non-trivial, non-Archimedean valuation. We relate the sheaf $\w\D$ of infinite order differential operators on smooth rigid $K$-analytic spaces to the algebra $\E$ of bounded $K$-linear endomorphisms of the structure sheaf. In the case of complex manifolds, Ishimura proved that the analogous sheaves are isomorphic. In the rigid analytic situation, we prove that the natural map $\w\D \to \E$ is an isomorphism if and only if the ground field $K$ is algebraically closed and its residue field is uncountable.
\end{abstract}
\maketitle
\tableofcontents
\setcounter{tocdepth}{1}  

\section{Introduction}
Ishimura proved in \cite{Ishimura} that the continuous endomorphisms of the structure sheaf of a complex manifold $X$, viewed as a sheaf of Fr\'{e}chet spaces over $\mathbb{C}$ in the classical topology correspond precisely with the infinite order differential operators on $X$. That is to say, he identified this sheaf of continuous endomorphisms with the sheaf of formal differential operators, whose total symbol in a small enough local chart $U\subset X$ around any point, defines a holomorphic function on $U \times \mathbb{C}^{\dim X}$. Some recent articles suggest defining a type of analytic geometry that would work in the same way over a general valuation field $K$ (Archimedean or not) or even a Banach ring. In this type of geometry, one constructs various Grothendieck topologies on the ``affine" objects. The structure sheaf then becomes a sheaf on the site defined by an analytic space and one can try, inspired by Ishimura's theorem, to define infinite order differential operators as the bounded endomorphisms of the structure sheaf over $K$. The structure sheaf in this case is considered in some category of sheaves with extra structure over $K$. This was the treatment used for complex manifolds by Prosmans and Schneiders in \cite{PrSch} who used this to partially interpret the Riemann-Hilbert correspondence as a derived Morita equivalence. This treatment uses their general technology of sheaves with values in a quasi-abelian category.  Their work is closely related to works of Mebkhout and Kashiwara on infinite order differential operators over complex manifolds; see for example \cite{Meb}. 

In this article, we look for an analogue of Ishimura's result for smooth rigid analytic spaces over a characteristic zero field $K$ which is complete with respect to a non-trivial non-Archimedean valuation. Our main result (Theorem \ref{Main}) gives further evidence that this ``endomorphisms'' definition of differential operators is sensible (at least when the ground field $K$ is large enough). Specifically, we prove that when $K$ is a (complete) non-trivially valued non-Archimedean field which is \emph{algebraically closed} and whose \emph{residue field is uncountable}, then the sheaf of bounded endomorphisms of the structure sheaf agrees with the sheaf of infinite order differential operators defined by Wadsley and the first author in \cite{DCapOne}.  

To reassure the reader that in the non-Archimedean setting there are many examples of such 'large' fields, we give the example of the field of \emph{Hahn series}
\[\mathbb{C}((\mathbb{Q}))= \{f= \underset{\gamma \in \mathbb{Q}}\sum a_{\gamma} x^{\gamma} | a_{\gamma} \in \mathbb{C}, \text{support of} \ \ f \ \ \text{well ordered}\}.\]
The non-Archimedean valuation $v$ on $\mathbb{C}((\mathbb{Q}))$ is given by
\[v\left(\underset{\gamma \in \mathbb{Q}}\sum a_{\gamma} x^{\gamma}\right) = \min\{ \gamma : a_\gamma \neq 0 \}, \]
and the residue field of $\mathbb{C}(\mathbb{Q})$ with respect to this valuation is $\mathbb{C}$. A mixed characteristic example can be found by starting with an arbitrary perfect field $k$ of characteristic $p > 0$ and taking $K$ to be the completion of the algebraic closure of the field of fractions of its ring of $p$-typical Witt vectors. The residue field of $K$ is then the algebraic closure of $k$, which is uncountable whenever $k$ is. The field of Puiseux series over the algebraic closure of a finite field gives an example of an uncountable, algebraically closed field of positive characteristic.
\subsection{Acknowledgements}
We would like to thank Kobi Kremnizer for suggesting that the kind of results contained in this article could exist. We would also like to thank Simon Wadsley and Thomas Bitoun for their interest in this work.
\section{Background}
\subsection{Sheaves of Banach spaces on rigid analytic varieties}\label{ShvBan}
Let $X$ be a smooth rigid analytic variety over a field $K$ complete with respect to a non-trivial non-Archimedean valuation. We denote by $X_{w}$ the category whose objects are the affinoid subdomains of $X$ and whose morphisms are the inclusions. This category is equipped with a Grothendieck topology (the weak topology) whose covers are given by collections of affinoid subdomains of $X$ which cover $X$ set-theoretically, and which locally on affinoid subdomains admit finite subcoverings. This topology on the category $X_w$ defines the rigid site of $X$. 

Let $L$ be another field which is complete with respect to a non-trivial non-Archimedean valuation. Recall the category of \emph{non-Archimedean Banach spaces over $L$} discussed at length in \cite{BK}. The morphisms between two such spaces $V,W$ are the \emph{bounded $L$-linear maps}, that is, the $L$-linear maps ${f : V \to W}$ for which there exists $C > 0$ such that $|f(v)|_W \leq C \cdot |v|_V$ for all $v \in V$. We denote the $L$-vector space of such maps by $\Hom_{L}(V,W)$. The category of \emph{pre-sheaves of $L$-Banach spaces} on the rigid site of $X$ is simply the category of functors $X_{w}^{\op}\to \ttBan_{L}$.  The category of \emph{sheaves of $L$-Banach spaces} on the rigid site of $X$ is the full category of pre-sheaves  of $L$-Banach spaces consisting of objects $P$ such that for any $V \in \ttBan_{L}$, $U \mapsto \Hom_L (V, P(U))$ is a sheaf of sets on the rigid site of $X$. This is equivalent to the following condition: for any finite admissible affinoid cover $\coprod U_i \to U$ of an admissible affinoid $U$ of $X$, there is a strict exact sequence 
\[0 \to P(U) \to \prod P (U_i) \to \prod P (U_{ij}).
\]
The category of Banach spaces is not abelian, instead it is quasi-abelian. We do not enter here into a general description of the category of sheaves on a site with values in quasi-abelian category, as we will need only the most basic structures in this article. A study of sheaves valued in quasi-abelian categories over topological spaces was initiated by Schneiders in \cite{Sch}.

The morphisms in the category of sheaves of $L$-Banach spaces on $X_{w}$ between $\mathcal{S}$ and $\mathcal{T}$ are given by  $\Hom_{X_{w}}(\mathcal{S},\mathcal{T}) = \ker (r)$, where 
\[r : \prod_{U \in X_w}\Hom_{L}(\mathcal{S}(U), \mathcal{T}(U)) \longrightarrow  \hspace{-0.6cm} \prod_{V, W \in X_w, W \subset V}\Hom_{L}(\mathcal{S}(V), \mathcal{T}(W))\] 
sends $f= (f_{U})_{U \in X_w}$ to the object $r(f)$  whose components are given by \[r(f)_{V,W} = \sigma_{W\subset V} \circ f_{V} - f_{W}  \circ \tau_{W\subset V}.
\]
Here $\sigma_{W\subset V}:\mathcal{S}(V) \to \mathcal{S}(W)$ and $\tau_{W\subset V}:\mathcal{T}(V) \to \mathcal{T}(W)$ are the restriction maps. We often think of an element $\psi \in \Hom_{X_{w}}(\mathcal{S},\mathcal{T})$ in terms of its ``tuple of components" $(\psi_{U})_{U \in X_w}$, which commute with the restriction maps:
\[\sigma_{W\subset V} \circ \psi_V = \psi_{W}  \circ \tau_{W\subset V}\]
whenever $W \subset V$ are objects of $X_w$. 

\subsection{Bornological spaces}\label{BornIntro} Motivated by the possible applications hinted at in the introduction, we would like to find a closed, symmetric, monoidal, quasi-abelian category which contains the category of $L$-Banach spaces fully faithfully, and is large enough to contain objects like $\Hom_{X_{w}}(\mathcal{S},\mathcal{T})$ described above. The category of complete locally convex topological $L$-vector spaces is not appropriate, because its natural symmetric monoidal structure does not extend to a \emph{closed} symmetric monoidal structure --- see \cite[Example 2.7]{Meyer2008}. Instead, we will work with the category of \emph{complete, convex, bornological vector spaces} because it satisfies all of these desiderata, as well as being fairly easy to work with. The nice thing about this category is that the categorical limit over a diagram is described as an explicit simple bornological structure on the limit of the same diagram in the category of vector spaces: see   \cite[Remark 3.44]{BaBe} for more details.  

A \emph{bornological space} is a set together with the data of a collection of distinguished subsets which are called the \emph{bounded subsets}. They are required to cover the underlying set, to be stable under passing to subsets and to be stable under finite unions. Any Fr\'{e}chet space (and, in particular, Banach space) can be considered a complete, convex, bornological space by taking the bounded subsets to be those which are bounded with respect to the metric. We will not give more details here, and instead refer the interested reader to \cite{BaBe},  \cite{BaBeKr} and the references given therein, but we will explain the natural bornological structure on the set $\Hom_{X_{w}}(\mathcal{S},\mathcal{T})$. 

If $V$ and $W$ are $L$-Banach spaces, then the space of bounded $L$-linear maps $\Hom_L(V,W)$ is again an $L$-Banach space, equipped with its operator norm. We will denote the induced bornological structure on this space by $\underline{\Hom}_L(V,W)$.  This construction globalises easily as follows: if $X$ is a smooth rigid analytic space over $K$, then whenever $\mathcal{S}$ and $\mathcal{T}$ are sheaves of $L$-Banach spaces on $X_{w}$ we have a complete convex bornological space $\underline{\Hom}_{X_w}(\mathcal{S},\mathcal{T}) = \ker (r)$ where
\begin{equation}\label{inthom}
\prod_{U \in X_w}\underline{\Hom}_{L}(\mathcal{S}(U), \mathcal{T}(U))\stackrel{r}\to \hspace{-0.6cm} \prod_{V, W \in X_w, W \subset V}\underline{\Hom}_{L}(\mathcal{S}(V), \mathcal{T}(W))
\end{equation}
and the products and kernels are calculated in the category of complete convex bornological spaces over $L$. The bounded subsets $\Hom_{X_w}(\mathcal{S},\mathcal{T}) $ (which define the bornology of $\underline{\Hom}_{X_w}(\mathcal{S},\mathcal{T}) $) are those subsets of  $\Hom_{X_w}(\mathcal{S},\mathcal{T}) $ whose image under the projection to $\Hom_{L}(\mathcal{S}(U), \mathcal{T}(U))$ are bounded in the operator norm, for each $U\in X_w$.  The underlying sets of such products and kernels in this category agree with those calculated in the category of vector spaces and so the underlying set of $\underline{\Hom}_{X_w}(\mathcal{S},\mathcal{T})$ is just $\Hom_{X_w}(\mathcal{S},\mathcal{T})$.
Explicitly, $\underline{\Hom}_{X_w}(\mathcal{S},\mathcal{T})$ is the set of collections of bounded linear maps from $\mathcal{S}(U)$ to  $\mathcal{T}(U)$ indexed by $U\in X_w$ that are compatible with the restrictions. In other words,  $\underline{\Hom}_{X_w}(\mathcal{S},\mathcal{T})$ is the set $\Hom_{X_w}(\mathcal{S},\mathcal{T})$ equipped with the bornology that is defined by letting the bounded subsets of  $\Hom_{X_w}(\mathcal{S},\mathcal{T})$ be those subsets for which the projection to each Banach space $\underline{\Hom}_{L}(\mathcal{S}(U), \mathcal{T}(U))$ is bounded.

Define $\underline{\mathcal{H}om}_{L}(\mathcal{S},\mathcal{T})$ to be the sheaf of bounded $L$-linear morphisms between $\mathcal{S}$ and $\mathcal{T}$. This is a sheaf of complete convex bornological spaces over $L$, whose value on an affinoid $U$ is given by
\[\underline{\mathcal{H}om}_{L}(\mathcal{S},\mathcal{T})(U) = \underline{\Hom}_{U_w}(\mathcal{S}|_{U},\mathcal{T}|_{U}).
\] 
Define $\underline{\mathcal{E}nd}_{L}(\mathcal{S})= \underline{\mathcal{H}om}_{L}(\mathcal{S},\mathcal{S})$ --- this is a sheaf of complete convex bornological algebras over $L$. The category of sheaves of $L$-Banach algebras (not-necessarily commutative) is similarly defined to be the full subcategory of functors from $X_{w}^{\op}$ to the category of $L$-Banach algebras which satisfy the same sheaf requirement. 

The rest of this paper will be concerned with the following explicit example. Whenever $X$ is a smooth rigid analytic space over $K$ and $U$ is an affinoid subdomain of $X$, the affinoid algebra $\O(U)$ carries a natural $K$-Banach algebra structure with respect to the \emph{supremum seminorm} on $U$, given by
\[ |f|_U := \sup \{ |f(x)| : x \in U \}.\] 
In this way, the structure sheaf $\O_X$ naturally becomes a sheaf of $K$-Banach spaces on $X_w$, and we therefore have at our disposal the sheaf
\[\mathcal{E}_X := \underline{\mathcal{E}nd}_{K}(\mathcal{O}_X)\]
of complete convex bornological $K$-algebras on $X_w$.

\subsection{Infinite order differential operators}
We now summarize some of the paper \cite{DCapOne}. In that article it was assumed that the valuation on the ground field $K$ was \emph{discrete}, whereas in this article we weaken this hypothesis and only demand that the valuation is non-trivial. As stated in \cite{DCapOne} many of the results in that article hold in this greater generality and this includes the results which we need. 

Fix a non-zero element $\pi \in K$ such that $|\pi|<1$. We assume that the reader is familiar with the definition of a \emph{Lie-Rinehart algebra} or more precisely an \emph{$(R, A)$-Lie algebra} where $R$ is a commutative ring and $A$ is a commutative $R$-algebra. It is a pair $(L, \rho)$ where $L$ is an $R$-Lie algebra and an $A$-module, and $\rho:L \to \Der_{R}(A)$ is an $A$-linear $R$-Lie algebra homomorphism satisfying a natural axiom.  For each $(R,A)$-Lie algebra $L$ there is an associated associative $R$-algebra $U(L)$ generated by $A$ and $L$ subject to some obvious relations, called the \emph{enveloping algebra} of $L$. 

Let $X$ be a smooth $K$-affinoid variety so that $\T(X)$ denotes the space $\Der_K \O(X)$ of $K$-linear derivations of $\O(X)$. The unit ball of the Banach algebra $\O(X)$ is the subring of \emph{power-bounded elements} of $\O(X)$, denoted $\O(X)^\circ$. We say that $\L$ is an $\O(X)^\circ$-\emph{lattice} in $\T(X)$ if it is finitely generated as an $\O(X)^\circ$-module, and spans $\T(X)$ as a $K$-vector space. We say that $\L$ is an $\O(X)^\circ$-\emph{Lie lattice} if in addition it is a sub-$(K^\circ,\O(X)^\circ)$-Lie algebra of $\T(X)$; it is not hard to see that $\O(X)^\circ$-Lie lattices always exist. For any $\O(X)^\circ$-Lie lattice $\L$, we write $\h{U(\L)}$ to denote the $\pi$-adic completion of $U(\L)$ and we write $\hK{U(\L)}$ to denote the Noetherian $K$-Banach algebra $K\otimes_{K^\circ}\h{U(\L)}$. The algebra of \emph{infinite order differential operators} on $X$ is
\[\w\D(X) := \invlim \hK{U(\pi^{n} \L)}\]
for any choice of $\O(X)^\circ$-Lie lattice $\L$ in $\T(X)$. Note that $\w\D(X)$ is naturally a $K$-Fr\'{e}chet algebra, with the Banach norms of the Banach $K$-algebras $\hK{U(\pi^n \L)}$ providing a countable family of seminorms on $\w\D(X)$ that define its Fr\'echet topology. It is shown in \cite[Section 6.2]{DCapOne} that this does not depend on the choice of $\L$.  

The algebra $\D(X) = U(\T(X))$ of finite order differential operators is dense in $\w{\D}(X)$. These constructions sheafify to define a sheaf of $K$-algebras $\D$ on $X_w$ which is dense in the sheaf of $K$-Fr\'{e}chet algebras $\w{\D}$ on $X_w$.

\section{Infinite order differential operators as endomorphisms}
In the rest of the article, $K$ will be a non-trivially valued, non-Archimedean valuation field of characteristic zero. The following elementary Lemma is fundamental to our constructions.

\begin{lem}\label{BddDers} Let $X$ be a $K$-affinoid variety. Then every $K$-linear derivation of $\O(X)$ is automatically \emph{bounded}.
\end{lem}
\begin{proof} This follows from the proof of \cite[Theorem 3.6.1]{FvdPut}; see also the discussion in \cite[Section 2.4]{DCapOne}.
\end{proof}

\begin{defn}
Let $V$ be a $K$-Banach space, and let $d \geq 0$. We say that a family of vectors $(a_\alpha)_{\alpha \in \mathbb{N}^d}$ in $V$ is {\it rapidly decreasing} if 
\[a_\alpha / \pi^{r |\alpha|} \to 0 \qmb{as} |\alpha| \to \infty \qmb{for any} r \in \mathbb{N}.\]
\end{defn}

\begin{rmk}\label{rmk:Equiv} The following conditions are equivalent:
\be \item $(a_\alpha)_{\alpha \in \mathbb{N}^d}$ is rapidly decreasing,
\item there exists $r_0 \in \mathbb{N}$ such that $a_\alpha / \pi^{r |\alpha|} \to 0$ as $|\alpha| \to \infty$ for all $r \in \mathbb{N}$ with $r \geq r_0$,
\item $\sup_{\alpha \in \mathbb{N}^d} |a_\alpha / \pi^{r |\alpha|}| < \infty$ for all $r \in \mathbb{N}$,
\item there exists $r_0 \in \mathbb{N}$ such that $\sup_{\alpha \in \mathbb{N}^d} |a_\alpha / \pi^{r |\alpha|}| < \infty$ for all $r \in \mathbb{N}$ with $r\geq r_0$.
\ee
\end{rmk}

We will always write $\mathbb{D}^d := \Sp K\langle x_1,\ldots, x_d\rangle$ to denote the $d$-dimensional polydisc over $K$. 

\begin{lem}\label{lem:symbol}Let $X$ be a smooth $K$-affinoid variety equipped with an \'etale morphism $g:X \to \mathbb{D}^d$. Let $\partial_i \in \T(X)$ be the canonical lifts of the standard vector fields $\frac{d}{dx_i}\in \T(\mathbb{D}^d)$. Then 
\[\w\D(X) = \left\{ \sum_{\alpha \in \mathbb{N}^d} a_\alpha \partial^\alpha : a_\alpha \in \O(X)\ \  \text{and} \ \ (a_\alpha)_{\alpha \in \mathbb{N}^d} \ \ \text{is rapidly decreasing} \right\}.\] The Fr\'{e}chet structure on $\w\D(X)$ can be defined by the family of
semi-norms \[\left|\sum_{\alpha \in \mathbb{N}^d} a_\alpha \partial^\alpha\right|_R= \underset{\alpha \in \mathbb{N}^d}\sup |a_\alpha|R^{\alpha}\] 
for sufficiently large real numbers $R$.\end{lem}
\begin{proof} Since $g : X \to \mathbb{D}^d$ is \'etale, $\{\partial_1, \dots, \partial_d \}$ forms an $\O(X)$-module basis for $\T(X)$. They need not preserve $\A:= \O(X)^\circ$ in general, but because they are bounded by Lemma \ref{BddDers}, we can find $N$ large enough so that the $\pi^{N}\partial_i$ do all preserve $\A$. Define $\L$ as the $\A$-submodule of $\T(X)$  generated by the bounded derivations $\{\partial_1, \dots, \partial_d \}$. Then $[\pi^N \L, \A] \subseteq \A$ by construction so we find that $\pi^N \L$ is a $(K^\circ,\A)$-Lie algebra which is in addition free of finite rank as an $\A$-module. Therefore, by \cite[Definition 9.3, Theorem 9.3, Definition 8.1 and Definition 6.2]{DCapOne}, we find that $\w\D(X)$ is canonically isomorphic to the inverse limit (over $r \geq N$) of the Banach algebras $\hK{ U( \pi^r \L) }$. 

Now, for any $r \geq N$, $\gr U(\pi^r \L) = S( \pi^r \L)$ by a theorem of Rinehart --- \cite[Theorem 3.1]{Rinehart} --- so $\gr U(\pi^r \L)$ is the commutative polynomial ring in $n$ variables over $\A$. These variables ${\zeta^{(r)}_1,...,\zeta^{(r)}_d}$ are compatible in the sense that the obvious (injective) map $S( \pi^s \L) \to S( \pi^r \L)$ for $s \geq r$ will send $\zeta^{(s)}_i$ to $\pi^{s-r} \zeta^{(r)}_i$ for any $i=1,...,n$. It follows that $U(\pi^r \L)$ consists of non-commutative polynomials in the $\pi^r \partial_i$ over $\A$. More precisely, every element of $U(\pi^r \L)$ can be written uniquely as a finite sum $\sum_{\alpha \in \mathbb{N}^d} a_\alpha (\pi^r \partial_1)^{\alpha_1}\cdots(\pi^r \partial_d)^{\alpha_d}$ with the $a_\alpha \in \A$. It also follows from this that the Banach completion $\hK{U(\pi^r \L)}$ is the non-commutative Tate-algebra over $\O(X) = \A_K$ in these variables: every element of $\hK{U(\pi^r \L)}$ can be written uniquely as a convergent power series $\sum_{\alpha \in \mathbb{N}^d} a_\alpha (\pi^r \partial_1)^{\alpha_1} ... (\pi^r \partial_d)^{\alpha_d}$ with the $a_\alpha$ now in $\O(X)$, tending to zero as $|\alpha| \to \infty$. By rewriting these sums as formal power series in $\partial_1,...,\partial_d$, we obtain the statement of the lemma.
\end{proof}

Whenever $V$ is an affinoid subdomain of a $K$-affinoid variety $U$, $\O(V)$ is naturally an abstract $\D(U)$-module. Because $\D(U)$ is generated as a $K$-algebra by $\O(U)$ and $\T(U)$, and because $\T(U)$ acts by \emph{bounded} derivations on $\O(U)$ by Lemma \ref{BddDers}, we thus obtain a $K$-algebra homomorphism
\[ \D(U) \longrightarrow \End_K \O(V).\]

\begin{lem}\label{DCapAction} Let $U$ be a smooth $K$-affinoid variety, and let $V$ be an affinoid subdomain of $U$. Then the $K$-algebra homomorphism $\D(U) \to \End_K \O(V)$ extends uniquely to a continuous $K$-algebra homomorphism
\[ \rho(U,V) : \w\D(U) \to \End_K \O(V)\]
such that, for any $W \in V_w$ and $a \in \w\D(U)$, the following diagram commutes:
\[ \xymatrix{  \O(V) \ar[d]\ar[rr]^{\rho(U,V)(a)} && \O(V) \ar[d] \\ \O(W) \ar[rr]_{\rho(U,W)(a)} && \O(W). }\]
\end{lem}
\begin{proof} 
The operator norm on $\End_K \O(V)$ is given by
\[ ||\phi|| = \sup \left\{ \frac{|\phi(f)|_V}{|f|_V} : f \in \O(V) \backslash \{0\}\right\}.\]
Because the unit ball of $\O(V)$ with respect to the supremum norm is $\O(V)^\circ$, the unit ball in $\End_K \O(V)$ with respect to the operator norm is 
\[\Hom_{K^\circ}( \O(V)^\circ, \O(V)^\circ ),\]
the space of \emph{all} $K^\circ$-linear maps $\O(V)^\circ \to \O(V)^\circ$. Note that because every open neighbourhood of zero in $\O(V)^\circ$ contains a $\pi$-power multiple of $\O(V)^\circ$, \emph{every} $K^\circ$-linear map $\O(V)^\circ \to \O(V)^\circ$ is automatically continuous. 

Now, using the proof of \cite[Lemma 7.6(b)]{DCapOne}, we can find an $\O(U)^\circ$-Lie lattice $\L$ in $\T(U)$ whose action on $\O(V)$ stabilises $\O(V)^\circ$. Because $\O(V)^\circ$ is $\pi$-adically complete, it follows that the induced action of $U(\L)$ on $\O(V)^\circ$ extends uniquely to the $\pi$-adic completion $\h{U(\L)}$. After tensoring over $K^\circ$ with $K$, we obtain in this way a norm-preserving $K$-Banach algebra homomorphism
\[ \hK{U(\L)} \longrightarrow \End_K \O(V),\]
which extends $\D(U) \to \End_K \O(V)$. We now define
\[\rho(U,V) : \w\D(U) \to \End_K \O(V)\]
to be the restriction of this homomorphism to $\w\D(U)$ along the canonical map $\w\D(U) \to \hK{U(\L)}$.  Because the set of $\O(U)^\circ$-Lie lattices stabilising $\O(V)^\circ$ is stable under intersections, it is easy to see that $\rho(U,V)$ does not, in fact, depend on the choice of $\L$. 

Finally, choose an $\O(U)^\circ$-Lie lattice $\L$ in $\T(U)$ which stabilises \emph{both} $\O(V)^\circ$ and $\O(W)^\circ$; then the restriction map $\O(V)^\circ \to \O(W)^\circ$ becomes $\h{U(\L)}$-linear. Therefore, the restriction map $\O(V) \to \O(W)$ is $\w\D(U)$-linear and the diagram in the statement of the Lemma commutes.\end{proof}

\begin{lem}\label{DcapAndEareSheaves} Let $X$ be a smooth rigid $K$-analytic space. Then $\w\D_X$ and $\E_X$ are sheaves of $K$-algebras on $X_w$.
\end{lem}
\begin{proof} The assertion about $\w\D_X$ was established in \cite[Theorem 8.1]{DCapOne} in the case where the ground field $K$ is discretely valued. However, the proof does not use this assumption on $K$ and therefore works in the stated generality. The discussion in $\S \ref{ShvBan}$ shows that $\E_X$ is also a sheaf on $X_w$.\end{proof}

\begin{cor} For any smooth rigid $K$-analytic space $X$, there is a homomorphism of sheaves of $K$-algebras on $X_w$
\[ \rho_X: \w\D_X \to \E_X\]
such that if $U \in X_w$, the $K$-algebra map $\rho_X(U) : \w\D(U) \to \E(U)$ is given by
\[ \rho_X(U)(a) = ( \rho(U,V)(a) )_{V \in U_w}.\]
\end{cor}
\begin{proof} By Lemma \ref{DCapAction}, we see that for any $U \in X_w$ and any $a \in \w\D(U)$, the element $( \rho(U,V)(a) )_{V \in U_w}$ lies in the kernel of the map $r$ described in $\S \ref{BornIntro}$ and so gives an element in $\Hom_{U_w} (\O_U, \O_U)$. By the definitions given in $\S \ref{ShvBan}$ and $\S \ref{BornIntro}$, 
\[ \Hom_{U_w}( \O_U, \O_U ) = (\underline{\mathcal{E}nd}_K \O_X)(U) = \E_X(U),\]
so we have defined a $K$-algebra homomorphism $\rho_X(U) : \w\D(U) \to \E_X(U)$. The construction of the restriction maps in the sheaf $\w\D$ explained in \cite[\S 6.3]{DCapOne} implies that these $\rho_X(U)$ commute with the restriction maps in the sheaves $\w\D_X$ and $\E_X$.
\end{proof}
Here is our main theorem. The proof will appear in Propositions \ref{btoa} and \ref{atob}.
\begin{thm}\label{Main} Let $K$ be a field of characteristic zero, complete with respect to a non-trivial non-Archimedean valuation. Then the following are equivalent:
\be\item The homomorphism of sheaves of $K$-algebras on $X_w$
\[\rho_X : \w\D_X \to \E_X\]
is an isomorphism for all smooth rigid $K$-analytic spaces $X$,
\item The ground field $K$ is algebraically closed, and its residue field $k$ is uncountable.
\ee\end{thm}
\begin{rmk}\label{rmk:resalg} \be \item If $\Char K = p > 0$ then $\rho_X$ need not be injective: for example, $\rho_{\mathbb{D}}$ annihilates $\partial^p$ where $\partial := \frac{d}{dx} \in \T(\mathbb{D})$, but $\partial^p$ is non-zero in $\w\D(\mathbb{D})$. This is why we restrict to the case $\Char K = 0$ throughout.

\item The assumption that $K$ is algebraically closed automatically forces its residue field $k$ to be algebraically closed as well. This is because any root of any monic lift of a monic non-constant polynomial in $k[X]$ to $K^\circ[X]$ already lies in $K^\circ$ because $K$ is algebraically closed and $K^\circ$ is integrally closed, so its image in $k$ is a root of our original polynomial.
\item The condition in Theorem \ref{Main} (b) that $K$ is algebraically closed will be used directly in Lemma \ref{lem:dense}. The implication that its residue field is algebraically closed will be used in Corollary \ref{cor:cExists}, Corollary \ref{EtaCor}, Lemma \ref{EtaRho}, and Theorem \ref{thm:DefFun}.
\ee\end{rmk}

\section{Proof of Theorem \ref{Main}}
We begin the proof of Theorem \ref{Main} with four general statements that should be skipped on a first reading.

\begin{lem}\label{FaithfulFlats}
Let $A$ be a domain and let $M$ be a non-zero flat $A$-module. Then $\Ann_A(M) = 0$.
\end{lem}
\begin{proof}
Suppose $a \in A$ is non-zero. Then the multiplication-by-$a$ map $a_A : A \to A$ is injective because $A$ is a domain. Since $M$ is flat, the multiplication-by-$a$ map $a_M : M \to M$ is injective. Since $M$ is non-zero, $a_M(M)$ is non-zero. Hence $a \notin \Ann_A(M).$ \end{proof}

\begin{prop}\label{InjRestrs}
Let $Z$ be a connected, smooth, $K$-affinoid variety, and let $Y$ be a non-empty affinoid subdomain of $Z$. Then the restriction map $\O(Z) \to \O(Y)$ is injective.
\end{prop}
\begin{proof}
Since $Z$ is smooth, the local rings $\O_{Z,z}$ of its structure sheaf are geometrically regular by \cite[Lemma 2.8]{BLR3}. By \cite[Proposition 7.3.2/8]{BGR} we know that the (usual) localisations $\O(Z)_{\mathfrak{m}}$ at all maximal ideals $\mathfrak{m}$ of $\O(Z)$ are regular local rings. Hence these local rings are domains by \cite[Theorem 11.22 and Lemma 11.23]{AMac}. Since $\O(Z)$ is Noetherian and since $Z$ is connected, it follows from \cite[Theorem 168]{KapCommRings} that $\O(Z)$ is a domain.

Since $Y$ is non-empty, $\O(Y)$ is non-zero. On the other hand $\O(Y)$ is a flat $\O(Z)$-module (in the purely algebraic sense) by \cite[Corollary 7.3.2/6]{BGR}. Since $\O(Z)$ is a domain, it follows from Lemma \ref{FaithfulFlats} that $\Ann_{\O(Z)} \O(Y) = 0$. But this annihilator is precisely the kernel of the map $\O(Z) \to \O(Y)$, and the result follows. \end{proof}

\begin{lem}\label{lem:tExists}
Let $Y$ be an irreducible, reduced, affine variety of finite type over an uncountable, algebraically closed field $k$, and let $a_0, a_1, a_2, \ldots$ be a sequence of non-zero elements of $\O(Y)$. Then there is $t \in Y(k)$ such that $a_i(t) \neq 0$ for all $i \geq 0$.
\end{lem}
\begin{proof}
Suppose first that $Y$ is the affine $n$-space $\mathbb{A}^n$ so that $\O(Y)$ is the polynomial algebra $k[y_1,\ldots,y_n]$. Proceed by induction on $n$, the case $n = 0$ being true vacuously.

Let $a_{ij} \in k[y_1,\ldots,y_{n-1}]$ be the coefficient of $y_n^j$ in $a_i$. By induction, we can find a point $(t_1,\ldots,t_{n-1}) \in k^{n-1}$ such that $a_{ij}(t_1,\ldots,t_{n-1}) \neq 0$ whenever $a_{ij} \neq 0$. Now $a_i(t_1,\ldots,t_{n-1},y_n)$ is a non-zero polynomial in $k[y_n]$ for all $i$, so it has only finitely many roots in $k$. Since $k$ is uncountable, we can find $t_n \in k$ such that $a_i(t_1,\ldots,t_{n-1},t_n) \neq 0$ for all $i\geq 0$ as required.

In general, use the Noether normalisation lemma to find a finite surjective morphism $f : Y \to \mathbb{A}^n$. Since $Y$ is irreducible and reduced, $\O(Y)$ is a domain. Since $\O(Y)$ is a finitely generated $\O(\mathbb{A}^n)$-module via $f^\sharp : \O(\mathbb{A}^n) \to \O(Y)$, $\O(Y)$ is integral over $\O(\mathbb{A}^n)$. Hence, the ideal $(f^\sharp)^{-1}(a_i \O(Y))$ is non-zero for all $i\geq 0$. So we can choose a non-zero element $b_i$ in this ideal for all $i\geq 0: f^\sharp(b_i) = a_i g_i$ for some $g_i \in \O(Y)$.  

Using the first part, choose $t \in \mathbb{A}^n(k)$ such that $b_i(t) \neq 0$ for all $i \geq 0$. Since $k$ is algebraically closed, the fibre $f^{-1}(t)(k)$ is non-empty. Pick any $y \in f^{-1}(t)(k)$. Then $a_i(y) g_i(y) = f^\sharp(b_i)(y) = b_i(f(y)) = b_i(t) \neq 0$ for all $i \geq 0$. Hence $a_i(y) \neq 0$ for all $i \geq 0$. 
\end{proof}
\begin{cor}\label{cor:cExists}Suppose the residue field $k$ of $K$ is uncountable and algebraically closed. Let $X$ be a reduced $K$-affinoid variety whose reduction is irreducible, and let $f_0, f_1,... \in \O(X)$ be a sequence of non-zero functions in $\O(X)$. Then there exists $c \in X(K)$ such that $|f_i(c)| = |f_i|_X$ for all $i\geq 0$.
\end{cor}
\begin{proof}
Let $Y = \Spec (\O(X)^\circ / \O(X)^{\circ\circ})$ be the reduction of $X$; it satisfies the hypotheses of  Lemma \ref{lem:tExists}. Because $X$ is reduced, we can find integers $m_0,m_1, \ldots \geq 1$, and non-zero scalars $\mu_0, \mu_1, \ldots \in K$, such that $|f_i|_X = |\mu_i|^{1/{m_i}}$ for all $i \geq 0$. Thus $f_i^{m_i}/\mu_i$ is an element of $\O(X)$ whose supremum norm is precisely $1$, and therefore its image $a_i \in \O(Y)$ is non-zero. Using Lemma \ref{lem:tExists}, choose $t \in Y(k)$ such that $a_i(k) \neq 0$ for all $i \geq 0$. 

Choose a lift $c \in X(K)$ of $t \in Y(k)$. Thus $|f_i(c)|^{m_i} / |\mu_i| = 1$ for all $i \geq 0$. But $|f_i|_X^{m_i} = |\mu_i|$, so we see that $|f_i|_X = |f_i(c)|$ for all $i \geq 0$ as required.  \end{proof}

\textbf{We assume until the end of the proof of Proposition \ref{RhoEta} below that $X$ is a smooth affinoid $K$-variety equipped with an \'etale morphism 
\[g : X \to \mathbb{D}^d.\] 
We also assume throughout that $K$ has characteristic zero.}

Using \cite[Lemma 2.4]{DCapOne}, we lift the standard vector fields $\frac{d}{dx_1}, \ldots, \frac{d}{dx_n}\in \T(\mathbb{D}^d)$ along the \'etale morphism $g$ to obtain $\partial_{1}, \dots, \partial_d \in \T(X)$. Note that we have \[\mathcal{T}(X)=\mathcal{O}(X)\partial_1 \oplus\mathcal{O}(X)\partial_2 \oplus \cdots \oplus \mathcal{O}(X)\partial_d.
\]
For every $\alpha\in\mathbb{N}^d$ we write $\partial^{\alpha}= \partial_{1}^{\alpha_{1}}\cdots \partial_d^{\alpha_d}$, $\alpha!= \alpha_{1}!\cdots \alpha_d!$ and $\binom{\alpha}{\beta}= \binom{\alpha_{1}}{  \beta_{1}} \cdots \binom{\alpha_d}{ \beta_d}$.  We write $\alpha \leq \beta$ to mean that $\alpha_i\leq \beta_i$ for $i=1, \dots, d$, and we define $|\alpha|= \alpha_{1} + \cdots +\alpha_d$. For any $U\in X_w$ we will abuse notation and write $x^{\alpha}$ denote the image of $x^{\alpha}=x_1^{\alpha_1}\cdots x_d^{\alpha_d}$ in $\O(U)$ under the composition $\mathcal{O}(\mathbb{D}^d) \stackrel{g^{\#}}\to \mathcal{O}(X) \to \mathcal{O}(U).$

Because $\Char K = 0$, we can make the following

\begin{defn}\label{DefnOfEta} Let $U \in X_w$, $\psi \in \End_K \O(U)$ and $\alpha \in \mathbb{N}^d$. Then we write
\[ \eta_\alpha(\psi) := \frac{1}{\alpha!}\sum_{\beta\leq\alpha} \psi(x^{\beta})\binom{\alpha}{ \beta }(-x)^{\alpha-\beta} \in \O(U).\]
Now, for every $\varphi\in \mathcal{E}(U)$ we can define the formal expression
\begin{equation}\label{eqn:EtaDef} \eta_{X}(U)(\varphi):=\sum_{\alpha \in \mathbb{N}^d} \eta_\alpha(\varphi(U)) \partial^{\alpha}\in \prod_{\alpha=0}^{\infty}\mathcal{O}(U)\partial^{\alpha}.
\end{equation}
\end{defn}

With these notations, we can formulate the following

\begin{lem}\label{EtaRestrictsWell} For every $U \in X_w$, every $\varphi \in \E(X)$ and every $\alpha \in \mathbb{N}^d$ we have
\[ \eta_\alpha( \varphi(X) )_{|U}= \eta_\alpha( \varphi(U) ).\]
\end{lem}
\begin{proof} This is immediate because $\varphi$ is a morphism of presheaves, and the restriction map $\O(X) \to \O(U)$ is a $K$-algebra homomorphism that carries $x^\beta \in \O(X)$ to $x^\beta \in \O(U)$ whenever $U \in X_w$.
\end{proof}

\begin{lem}\label{lem:invariant}Let $\psi : \O(X) \to \O(X)$ be a $K$-linear map, let $c_1,\ldots, c_d \in K$ and define $y_i := x_i - c_i$  for all $i$.  Then
\[\sum_{\beta \leq \alpha}\psi(x^\beta) \binom{\alpha}{\beta}   (-x)^{\alpha - \beta}= \sum_{\beta \leq \alpha}   \psi(y^\beta)\binom{\alpha}{\beta}(-y)^{\alpha - \beta}.
\]
\end{lem}
\begin{proof}Consider the $K$-linear map 
\[
Q := \mult\circ (\psi \otimes 1_{\O(X)}) : \O(X) \otimes \O(X) \to \O(X).
\]
The left, respectively, right, hand side of the equality we propose to show can be identified with $Q( (x \otimes 1 - 1 \otimes x)^\alpha)$, respectively, $Q( (y \otimes 1 - 1 \otimes y)^\alpha)$. However, 
\[ x_i \otimes 1 - 1 \otimes x_i = (y_i + c_i) \otimes 1 - 1 \otimes (y_i + c_i) = y_i \otimes 1 - 1 \otimes y_i\]
and hence $(x \otimes 1 - 1 \otimes x)^\alpha = (y \otimes 1 - 1 \otimes y)^\alpha$ for any $\alpha \in \mathbb{N}^d$. \end{proof}

\begin{lem}\label{lem:FactGrow} There is a real number $0 < \varpi \leq 1$ depending only on $K$ such that $\varpi^{|\alpha|} \leq |\alpha!|$ for all $\alpha \in \mathbb{N}^{d}$.
\end{lem}
\begin{proof}
Consider the prime subfield $\mathbb{Q} \subset K$. If the induced valuation on $\mathbb{Q}$ is trivial then we may take $\varpi=1$ as in this case $|\alpha!| = 1$ for any $\alpha \in \mathbb{N}^d$. 

Otherwise, by Ostrowski's Theorem there is a unique prime number $p$ such that $|p| < 1$ and $|p'| = 1$ for every prime $p' \neq p$, so that the topology induced by $K$ on $\mathbb{Q}$ is the $p$-adic topology. Define
\[\varpi := |p|^{\frac{1}{p-1}}.\]
Now, for the $p$-adic valuation $v_p$ on $\mathbb{Q}$ we have the estimate $v_p(m!) \leq \frac{m}{p-1}$ by \cite[Chapter II, Section 8.1, Lemma 1]{BourLGLA}. Since $|p| < 1$ we deduce that
\[|m!| = |p|^{v_p(m!)}\geq \varpi^m \qmb{for any} m \in \mathbb{N}.\]
Hence $|\alpha!|= |\alpha_1!|\cdots |\alpha_d!|\geq \varpi^{\alpha_1} \cdots \varpi^{\alpha_d} = \varpi^{|\alpha|}$ as claimed.
\end{proof}
\begin{thm}\label{thm:DefFun}Suppose that the residue field $k$ of $K$ is algebraically closed and uncountable. Then for every $\varphi \in \E(X)$, the family
\[ (\eta_\alpha(\varphi(X)))_{\alpha \in \mathbb{N}^d} \subset \O(X)\]
is rapidly decreasing.
\end{thm}
\begin{proof}
Write $\xi_\alpha := \eta_\alpha(\varphi(X)) \in \O(X)$, and let $S = \{\alpha \in \mathbb{N}^d | \xi_\alpha \neq 0\}$. The case when $S$ is finite is clear, so assume that $S$ is countably infinite. 

Assume first that the reduction  $Y = \text{Spec}  (\O(X)^\circ / \O(X)^{\circ\circ})$ is irreducible. Note that $X$ is reduced because it is assumed to be smooth, so using Corollary \ref{cor:cExists}, we may choose $c \in X(K)$ such that 
\[
   |\xi_\alpha|_X = |\xi_\alpha(c)| \qmb{for all} \alpha \in S.\]
  Choose a non-zero element $\pi \in K$ with $|\pi| < 1$, and fix $n \geq 0$. Consider the affinoid subdomain
  \[ X_n := \Sp \left( \O(X) \left\langle \frac{x_1 - g(c)_1}{\pi^n}, \frac{x_2 - g(c)_2}{\pi^n}, \ldots, \frac{x_d - g(c)_d}{\pi^n} \right\rangle\right)\]
of $X$, and let $y_i := x_i - g(c)_i$ for $i=1,\ldots, d$. Then
\[
\xi_\alpha |_{X_n} = \eta_\alpha( \varphi(X) )_{|X_n} = \eta_\alpha( \varphi(X_n) )=\frac{1}{\alpha!}\sum_{\beta\leq \alpha} (-y)^{\alpha - \beta} \binom{\alpha}{\beta} \varphi(X_n)(y^\beta) \]
by Lemma \ref{EtaRestrictsWell} and Lemma \ref{lem:invariant}. Now $|y_i|_{X_n} \leq |\pi^n|$ for each $i=1,\ldots,d$, and $c \in X_n(K)$ for all $n \geq 0$. Therefore
\begin{equation}\label{XiNormEstimate}
\begin{split}
  |\xi_\alpha|_X = |\xi_\alpha(c)| \leq |\xi_\alpha|_{X_n} &  \leq \frac{1}{|\alpha!|}\max_{\beta \leq \alpha} |y^{\alpha-\beta}|_{X_n}\cdot ||\varphi(X_n)|| \cdot |y^\beta|_{X_n} \\ & \leq \frac{||\varphi(X_n)||\cdot |\pi^{n|\alpha|}|}{|\alpha!|} 
  \end{split}
\end{equation}
for all $n \geq 0$ and all $\alpha \in S$.
Choose $N \in \mathbb{N}$ large enough so that $|\pi|^{N}\leq \varpi$. Then by Lemma \ref{lem:FactGrow}, 
\[ | \pi^{N |\alpha|} | = |\pi|^{ N |\alpha| }\leq \varpi^{|\alpha|} \leq |\alpha!|.\]
Therefore, applying (\ref{XiNormEstimate}) with $n$ replaced by $n + N + 1$, we obtain
\[\frac{ |\xi_\alpha|_X}{|\pi^{n|\alpha|}|} \leq \frac{ || \varphi(X_{n+N+1}) || \cdot | \pi^{N |\alpha|} | \cdot |\pi|^{|\alpha|}}{|\alpha!| } \leq || \varphi(X_{n+N+1}) || \cdot |\pi|^{|\alpha|}\]
for any $\alpha \in S$, and therefore also for any $\alpha \in \mathbb{N}^d$ since $|\xi_\alpha|_X = 0$ if $\alpha \notin S$. Because $|\pi| < 1$, the family of elements $(\xi_\alpha)$ in $\O(X)$ is rapidly decreasing:
\[\frac{ \xi_\alpha}{\pi^{n|\alpha|}} \to 0 \quad \mbox{as} \quad |\alpha| \to \infty, \qmb{for any} n \geq 0.\]
Returning to the general case, let $\{|\cdot|_1, |\cdot|_2,\dots, |\cdot|_r\}$ be the Shilov boundary of the Berkovich space associated to $X$. It follows from \cite[Proposition 2.4.4]{BerkovichFirstBook} that 
\[|a|_X = \max_{1 \leq i \leq r} |a|_i \qmb{for any} a \in \O(X).\] 
Furthermore, we can find elements $g_1,\ldots, g_r \in \O(X)$ of supremum norm $1$ such that if $X_i$ denotes the Laurent subdomain $X(1/g_i)$ of $X$ and $\rho_i : \O(X) \to \O(X_i)$ is the restriction map, then $|\rho_i(a)|_{X_i} = |a|_i$ for all $a \in \O(X)$, and $X_i$ has irreducible reduction for each $i = 1,\ldots, r$.

Now, $\rho_i(\xi_\alpha) = \eta_\alpha(\varphi(X))_{|X_i} = \eta_\alpha( \varphi(X_i))$ by Lemma \ref{EtaRestrictsWell}, and $\varphi|_{X_i}$ is an element of $\E(X_i)$. So by the above, we know that
\[|\rho_i(\xi_\alpha) / \pi^{n |\alpha|}|_{X_i} \to 0\]
for all $i = 1,\ldots, r$, and for all $n \geq 0$. Hence, for any $n \geq 0$,
\[|\xi_\alpha / \pi^{n|\alpha|}|_X = \max_{1 \leq i \leq r} |\rho_i(\xi_\alpha) / \pi^{n |\alpha|}|_{X_i} \to 0\]
as $\alpha \to \infty$ also. 
\end{proof}
\begin{cor}\label{EtaCor} Suppose that the residue field of $K$ is algebraically closed and uncountable. Then
\be \item for any $\varphi\in \mathcal{E}_{X}(U)$, the expression
\[ \eta_X(U)(\varphi) := \sum_{\alpha \in \mathbb{N}^d} \eta_\alpha(\varphi(U)) \partial^{\alpha}\]
defines an element of $\wideparen{{\mathcal{D}}}(U)$,
\item there is a well-defined morphism of $\O_X$-modules
\[ \eta_X : \E_X \to \w\D_X,\]
\item for every $\varphi \in \E(X)$, the \emph{total symbol}
\[T(\varphi):= \sum_{\alpha \in \mathbb{N}^d} \eta_\alpha( \varphi(X) ) \hspace{0.1cm} \zeta^\alpha\in \O(X)[[\zeta_1,\zeta_2,\ldots,\zeta_d]]
\]
defines a rigid analytic function on $X \times \mathbb{A}^d$,
\ee\end{cor}
\begin{proof} (a) This follows from Theorem \ref{thm:DefFun} and Lemma \ref{lem:symbol}.

(b) By part (a), for every $U \in X_w$ there is an $\O(U)$-linear morphism
\[ \eta_X(U) : \E_X(U) \to \w\D_X(U).\] 
Since $\varphi$ is a morphism of sheaves over $U_w$ we have $\varphi(U)(x^{\beta})|_{V}=\varphi(V)(x^{\beta})$ for any $V\in U_w$. Therefore, Definition \ref{DefnOfEta} shows that the $\eta_X(U)$ are compatible with restriction and so define a morphism of sheaves.

(c) For any $n \geq 0$, $T(\varphi)$ converges on $\O(X \times \Sp K \langle \pi^n \zeta_1, \ldots \pi^n \zeta_d \rangle)$ if and only if $\frac{ \xi_\alpha}{\pi^{n|\alpha|}} \to 0$ as $|\alpha| \to \infty$. So, by Theorem \ref{thm:DefFun}, $T(\varphi)$ converges everywhere on $X \times \mathbb{A}^d \cong T^\ast X$.\end{proof}
\begin{lem}\label{EtaRho} Suppose that the residue field of $K$ is algebraically closed and uncountable. Then $\eta_X \circ \rho_X = \text{id}_{\w\D}$.
\end{lem}
\begin{proof} We have show that $\eta_X(U) \circ \rho_X(U) = 1_{\w\D(U)}$ for any $U \in X_w$. Since the restriction of $g : X \to \mathbb{D}^d$ to $U \subset X$ is still \'etale, it will be sufficient to consider the case $U = X$ only. 

Let $a = \sum_{\alpha \in \mathbb{N}^d} a_\alpha \partial^\alpha \in \w\D(X)$ so that $(a_\alpha)_{\alpha \in \mathbb{N}^d}$ is a rapidly decreasing family of elements in $\O(X)$, and let $\psi := \rho_X(X)(a) = \rho(X,X)(a) \in \End_K \O(X)$. We must show that
\[ \eta_\gamma(\psi) = a_\gamma \qmb{for all} \gamma \in \mathbb{N}^d.\]
Now $\psi(x^\beta) = \rho(X,X)(a)(x^\beta) = \underset{\alpha \leq \beta}{\sum} a_\alpha \frac{\beta!}{(\beta - \alpha)!} x^{\beta - \alpha}$ for any $\beta \in \mathbb{N}^d$. Hence
\begin{equation}\label{eqn:combin}\begin{split} \gamma! \eta_\gamma(\psi) = & \sum_{\beta \leq \gamma} \psi(x^\beta)\binom{\gamma}{\beta} (-x)^{\gamma - \beta} = \\
& = \sum_{\beta\leq \gamma } \left( \sum_{\alpha\leq\beta }a_\alpha \frac{\beta!}{(\beta - \alpha)!} x^{\beta - \alpha} \right) \binom{ \gamma }{ \beta } (-x)^{\gamma-\beta} = \\
& = \sum_{\alpha \leq \gamma} \left( \sum_{\alpha \leq \beta \leq \gamma} \frac{\beta!}{(\beta - \alpha)!} \binom{\gamma}{\beta} (-1)^{\gamma - \beta} \right) a_\alpha x^{\gamma - \alpha}.
\end{split}
\end{equation}
For fixed $\alpha \leq \gamma$ and $\beta$ in the range $\alpha \leq \beta \leq \gamma$, make the change of variable $\epsilon := \beta - \alpha$. Then
\[ \frac{\beta!}{(\beta - \alpha)!} \cdot \binom{\gamma}{\beta} = \frac{\gamma!}{(\beta - \alpha)! (\gamma - \beta)! } = \frac{ \gamma!}{\epsilon! (\gamma - \alpha - \epsilon)!} = \binom{\gamma - \alpha}{\epsilon} \cdot \frac{\gamma!}{(\gamma - \alpha)!}.\]
Therefore, for any $\alpha \leq \gamma$ we have
\[\sum_{\alpha \leq \beta \leq \gamma} \frac{\beta!}{(\beta - \alpha)!} \binom{\gamma}{\beta} (-1)^{\gamma - \beta} =  \frac{\gamma!}{(\gamma - \alpha)!} \sum_{\epsilon \leq \gamma - \alpha} \binom{\gamma - \alpha}{\epsilon}(-1)^{\gamma - \alpha - \epsilon} = \gamma!\delta_{\alpha, \gamma},\]
and therefore $\gamma!\eta_\gamma(\psi) = \sum_{\alpha \leq \gamma} \gamma!\delta_{\alpha,\gamma} a_\alpha x^{\gamma - \alpha} = \gamma!a_\gamma$. Since $\Char K = 0$, we  conclude that $\eta_\gamma(\psi) = a_\gamma$ as required.
\end{proof}

Recall that we are assuming that $X$ is a smooth $K$-affinoid variety, equipped with an \'etale morphism $g : X \to \mathbb{D}^d$. 

\begin{lem}\label{lem:dense}Suppose that $K$ is algebraically closed. Then there exists an affinoid subdomain $t : Y \hookrightarrow X$ such that $(g \circ t)^{\#}:\mathcal{O}(\mathbb{D}^d)\to \mathcal{O}(Y)$ has dense image.
\end{lem}
\begin{proof}Pick any point $b\in X$ and let $a= g(b)\in \mathbb{D}^d$. Since $K=\overline{K}$, $F_a = F_{b} = \overline{K}$. See  \cite[Definition 7.4.4]{FvdPut} for the notation $F_{a}$. By  \cite[\S 8.1.3]{FvdPut} there exists a wide open neighbourhood $U$ of $a\in \mathbb{D}^d$ such that $g^{-1}(U)= V \coprod W$ where $V$ and $W$ are affinoid, $V$ is a wide open neighborhood of $b$ and $g|_{V}:V \to U$ is an open immersion. Choose any polydisk $T \subset g(V)$ which contains $a$, and let $Y=g^{-1}(T) \subset V$.  Then we have the commutative diagram
\begin{equation}
\xymatrix{Y \ar[r]^{\subset}\ar[d]^{\cong}_{g_{|Y}} &  V\ar[d]  \ar[r]^{\subset} \ar[dr]& V \coprod W \ar[r]^{\subset}  \ar[d] & X\ar[d]^{g} \\ T \ar[r]^{\subset} & g(V) \ar[r]^{\subset}  & U \ar[r]^{\subset} & \mathbb{D}^d.}
\end{equation}
Because $T$ is a Weierstrass subdomain of $\mathbb{D}^d$, \cite[Proposition 7.3.4/2]{BGR} implies that the image of the restriction map $\theta : \mathcal{O}(\mathbb{D}^d)\to \mathcal{O}(T)$ is dense. On the other hand, $(g_{|Y})^\sharp:\mathcal{O}(T)\to \mathcal{O}(Y)$ is an isomorphism because $g|_{V}: V \to U$ is an open immersion. Therefore, the image of $(g \circ t)^\sharp = t^\sharp \circ g^\sharp = (g_{|Y})^\sharp \circ \theta$ is dense, as claimed.
\end{proof}

\begin{prop}\label{EtaInj} Suppose that $K$ is algebraically closed and that $X$ is connected. Let $\varphi \in \E(X)$ be such that $\eta_\alpha(\varphi(X)) = 0$ for all $\alpha \in \mathbb{N}^d$. Then $\varphi(X) = 0$. 
\end{prop}
\begin{proof} We will first show by induction on $|\alpha|$ that $\varphi(X)(x^\alpha) = 0$ for all $\alpha \in \mathbb{N}^d$. This is true when $\alpha = 0$ because $\varphi(X)(1) = \eta_0( \varphi(X) ) = 0$. Now 
\[ \eta_\alpha(\varphi(X)) = \frac{1}{\alpha!}\sum_{\beta \leq \alpha} \varphi(X)( x^\beta) \binom{\alpha}{\beta} (-x)^{\alpha - \beta}\]
and $|\beta| < |\alpha|$ whenever $\beta \leq \alpha$ and $\beta \neq \alpha$. Therefore $\varphi(X)(x^\beta) = 0$ for any $\beta \leq \alpha$ with $\beta \neq \alpha$ by induction, and we deduce that $\varphi(X)(x^\alpha) = \alpha!\eta_\alpha(\varphi(X)) = 0$. This completes the induction. 

Next, applying Lemma \ref{lem:dense} to $g: X \to  \mathbb{D}^d$ gives us an affinoid subdomain $t : Y \hookrightarrow X$ such that $(g \circ t)^{\sharp}:\mathcal{O}(\mathbb{D}^d) \to \mathcal{O}(Y)$ has dense image.
Consider the following commutative diagram:
\begin{equation}
\xymatrix{\mathcal{O}(\mathbb{D}^d) \ar[dr]_{(g \circ t)^\sharp} \ar[r]^{g^\sharp}& \mathcal{O}(X)\ar[d]^{t^\sharp} \ar[r]^{\varphi(X)}& \mathcal{O}(X) \ar[d]^{t^\sharp} \\ &  \mathcal{O}(Y) \ar[r]_{\varphi(Y)} & \mathcal{O}(Y). 
}
\end{equation}
Since $\varphi(X) : \O(X) \to \O(X)$ is bounded, it is continuous. The restriction map $g^\sharp$ is also continuous, by \cite[Theorem 3.2.1(7)]{FvdPut}.  Because the polynomial algebra $K[x_1,\ldots, x_d]$ is dense in $\O(\mathbb{D}^d) = K \langle x_1,\ldots, x_d \rangle$ and because we saw above that $\varphi(X)$ kills $K[x_1,\ldots, x_d]$, we deduce that $\varphi(X) \circ g^\sharp$ is zero.

The commutativity of the diagram implies that $\varphi(Y) \circ (g\circ t)^{\#}=0$. Because $(g\circ t)^{\#}$ has dense image and $\varphi(Y)$ is continuous we can conclude that $\varphi(Y)=0$. Therefore $t^\sharp \circ \varphi(X) = \varphi(Y) \circ t^\sharp = 0$, but $t^\sharp : \O(X) \to \O(Y)$ is injective by Proposition \ref{InjRestrs} because $X$ is connected, so $\varphi(X) = 0$ as required.
\end{proof}

If $K$ is not algebraically closed, then Proposition \ref{EtaInj} breaks down: take $X$ to be $\Sp L$ for some non-trivial finite field extension $K$ of $L$; then any non-zero $K$-linear endomorphism of $L$ that is zero on $K$ will provide a counterexample.

\begin{prop}\label{RhoEta} Suppose that $K$ is algebraically closed and that its residue field $k$ is uncountable. Then $\rho_X \circ \eta_X = 1_{\E_X}.$
\end{prop}
\begin{proof} Note that $k$ is automatically algebraically closed by Remark \ref{rmk:resalg}(b). By Lemma \ref{EtaRho} we see that $\eta_X \circ (\rho_X \circ \eta_X - 1) = (\eta_X \circ \rho_X - 1) \circ \eta_X = 0.$ Thus it will be enough to show that $\eta_X$ is injective. 

Suppose that $\eta_X(U)(\varphi) = 0$ for some $U \in X_w$ and some $\varphi \in \E(U)$. Let $V$ be a \emph{connected} affinoid subdomain of $U$; then also $\eta_X(V)(\varphi_{|V}) = \eta_X(U)(\varphi)_{|V} = 0$ and hence $\eta_\alpha(\varphi(V)) = 0$ for all $\alpha \in \mathbb{N}^d$. Because the restriction of ${g : X \to \mathbb{D}^d}$ to $V$ is \'etale, it follows from Proposition \ref{EtaInj} that $\varphi(V) = 0$. Now $\varphi$ can be regarded as a morphism of \emph{sheaves} $\O \to \O$ on $U_w$; since every $V \in U_w$ a finite disjoint union of its connected components, we deduce that $\varphi(V) = 0$ for \emph{all} $V \in U_w$. Hence $\varphi = 0$ and $\eta_X(U)$ is injective for every $U \in X_w$. \end{proof}

We now return to full generality, and begin our proof of Theorem \ref{Main}.

\begin{prop}\label{prop:etaleDiskcovering}  Let $X$ be a smooth rigid analytic space. Then $X$ admits an admissible covering by a collection of affinoids $X_{i}$, such that there exist \'{e}tale morphisms $g_{i}:X_{i} \to \mathbb{D}^{n_i}$ to disks of various dimensions.
\end{prop}
\begin{proof} This follows from \cite[Proposition 2.7]{BLR3}.\end{proof}

We now prove the central result of this article. The (a) $\Rightarrow$ (b) part of Theorem \ref{Main} is proven in Proposition \ref{atob}.
\begin{prop}\label{btoa} (b) $\Rightarrow$ (a) from Theorem \ref{Main} holds.
\end{prop}

\begin{proof}By Lemma \ref{DcapAndEareSheaves}, $\w\D_X$ and $\E_X$ are sheaves on $X_w$. Therefore, the morphism $\rho_X$ is an isomorphism if and only if its restriction to each member of an admissible affinoid covering of $X$ is an isomorphism. By Proposition \ref{prop:etaleDiskcovering}, we may therefore assume that $X$ is an affinoid variety which admits an \'etale morphism $g : X \to \mathbb{D}^d$ to some $d$-dimensional polydisc. 

Assume that $K$ is algebraically closed, and that its residue field $k$ is uncountable. Then $k$ is also algebraically closed by Remark \ref{rmk:resalg}(b). Therefore $\eta_X$ is a well-defined $\O_X$-linear morphism $\E_X \to \w\D_X$ by Corollary \ref{EtaCor}(b), and it is a two-sided inverse to $\rho_X$ by Lemma \ref{EtaRho} and Proposition \ref{RhoEta}.
\end{proof}

Let us now fix a non-Archimedean field $K$ with non-trivial valuation whose residue field $k$ is \emph{countable}, for example $K$ could be the field $\mathbb{C}_{p}$ of $p$-adic complex numbers whose residue field is $\overline{\mathbb{F}_{p}}$.  Let 
\[X := \mathbb{D} = \Sp K\langle x \rangle\]
be the closed unit disc. Because $k$ is assumed to be countable, we may choose a set $\{\lambda_{\alpha} | \alpha \in \mathbb{N} \} \subset K^\circ$ of additive coset representatives for the maximal ideal $\mathfrak{m}$ in $K^\circ$. So we have 
\[K^\circ = \coprod_{\alpha=0}^{\infty}(\lambda_\alpha +\mathfrak{m}).\] 
Now, define $\partial^{(\alpha)} := \frac{\partial^\alpha}{\alpha!}$ for any $\alpha \in \mathbb{N}^d$, and consider the formal power series
\[\xi := \sum_{\alpha=0}^{\infty}\xi_{\alpha} \pi^\alpha \partial^{(\alpha)}_{x} \in \prod_{\alpha=0}^{\infty}K\langle x \rangle \partial^{(\alpha)}_{x}
\]
where \[\xi_{\alpha} := (x-\lambda_{0})^{\alpha^{2}}\cdots (x-\lambda_{\alpha})^{\alpha^{2}}.\]
We make two claims. 

\begin{claim}\label{Claim1} For all $Y \in X_{w}$, $(\xi_{\alpha} \pi^\alpha \partial^{(\alpha)}_{x|Y})_{\alpha=0}^{\infty}$ tends to zero in the Banach algebra $\underline{\End}_{K}(\O(Y))$ of bounded $K$-linear endomorphisms of $\O(Y)$. Therefore, $\xi$ defines a global section of $\E$.
\end{claim} 
\begin{claim}\label{XiNotInDCap} \[\xi \notin \im \left(\wideparen{{\mathcal{D}}}(X) \to  \prod_{\alpha=0}^{\infty}K\langle x \rangle\partial^{(\alpha)}_{x}\right)\]
\end{claim}
Together, these claims show that the map $\rho : \w\D\to \E$ is not an isomorphism over the one dimensional closed unit disk, when $K$ has countable residue field. 

We start with the proof of Claim \ref{Claim1}. First of all, we can assume that $Y$ is connected. Otherwise, $Y$ has a finite number of connected components $Y_i$ and $(\xi_{\alpha} \pi^\alpha \partial^{(\alpha)}_{x})_{\alpha=0}^{\infty}$ being a zero sequence in $\underline{\End}_{K}( \O(Y) )$  is equivalent to its image in $\underline{\End}_{K}(\O(Y_i))$ being a zero sequence for each $i$. The next step is to prove it in the case that $Y$ is contained in a closed disk $Z$ of radius \emph{strictly} less than $1$. 

Thus, assume that there exist $a, \rho \in \O_K$ with $|\rho| < 1$ such that $Y \subseteq Z := \Sp K\langle z \rangle$, where $z := \frac{x-a}{\rho}$. There exists a unique integer $\gamma \geq 0$ such that 
\[ |\epsilon| < 1 \quad\text{where} \quad \epsilon := a - \lambda_\gamma.\]
Because $x = a + \rho z$ and because $|z|_Z \leq 1$, we see that 
\[|x-\lambda_\gamma|_Z=|\epsilon+\rho z|_Z \leq \max\{ |\epsilon|, |\rho|  \} < 1.\]
Now let $\beta \in \mathbb{N}$. Because $|\lambda_\gamma - \lambda_\beta| = \delta_{\beta,\gamma}$ by construction, 
  \[
    |x-\lambda_{\beta}|_{Z} = |\lambda_\gamma - \lambda_\beta + (x - \lambda_\gamma)|_Z =  \left\{\begin{array}{lr}
        |\epsilon+\rho z|_Z & \text{if } \beta= \gamma\\
       1 & \text{otherwise. }
        \end{array}\right.
  \]
Hence, whenever $\alpha \geq \gamma$ we see that
\[   |\xi_{\alpha}|_Z
    = \prod_{\beta=0}^\alpha |x-\lambda_{\beta}|^{\alpha^{2}}_Z 
    = |\epsilon+\rho z|_Z^{\alpha^2}
    \leq (\max \{|\epsilon|,|\rho|\})^{\alpha^2}.
  \]
  Therefore, $(\xi_{\alpha|Z})$ is rapidly decreasing: for any $N \geq 0$, whenever $\alpha \geq \gamma$ we see that
  \[\log(|\xi_{\alpha}|_{Z} / |\pi|^{N\alpha}) \leq -(\log |\pi|)N \alpha+ \log(\max \{|\epsilon|,|\rho|\})\alpha^{2}
  \]
which tends to $-\infty$ as $\alpha \to \infty$ because $\log(\max \{|\epsilon|,|\rho|\})<0$. Therefore
\[ \lim\limits_{\alpha\to\infty} \frac{|\xi_\alpha|_Z}{|\pi|^{N\alpha}}=0 \quad \mbox{whenever}\quad N \geq 0\]
as claimed. Next, let $R > 0$ be the operator norm of the restriction of $\partial_x$ to $\O(Y)$. Then the operator norm of the restriction of $\partial_x^{(\alpha)} = \frac{\partial_x^\alpha}{\alpha!}$ to $\O(Y)$ is at most $R^\alpha / |\alpha!|$. Choose a positive integer $N$ such that $|\pi|^N < \varpi / R$. Then 
\[ \frac{R^\alpha}{|\alpha!|} \leq \frac{R^{\alpha}}{\varpi^{\alpha}} < \frac{1}{|\pi|^{N\alpha}} \qmb{for all} \alpha \in \mathbb{N} \]
by Lemma \ref{lem:FactGrow}. The restriction map $\O(Z) \to \O(Y)$ is contracting, and we deduce that
\[||\xi_\alpha \partial_x^{(\alpha)}||_Y \leq |\xi_\alpha|_Y \frac{R^\alpha}{|\alpha!|} \leq \frac{|\xi_{\alpha}|_Z}{|\pi|^{N\alpha}} \to 0\]
as $\alpha \to \infty$. Therefore $(\xi_{\alpha} \pi^\alpha \partial^{(\alpha)}_{x|Y})_{\alpha=0}^{\infty}$ is a zero sequence in the Banach algebra $\underline{\End}_{K}(\O(Y))$, as claimed.

  Suppose now there does \emph{not} exist a closed disk as above containing $Y\in X_w$. Recall \cite[Proposition 2.2.6]{FvdPut} which classifies the objects in $X_w$ and tells us that since $Y$ is not contained in any closed disc of radius strictly less than one, we have 
 \[Y = X-\left(B(a_1, |\tau_1|) \cup \cdots \cup B(a_n, |\tau_n|) \right)\]
  for some $a_1, \dots, a_n \in X(K)$ and $0<|\tau_1|, \dots, |\tau_n| \leq 1$.  Here $B(a_i,|\tau_i|)$ denotes the open rigid analytic disc of radius $|\tau_i|$ around $a$. We proceed by induction on $n\geq 0$. First, we observe that
\begin{equation}\label{DeltaEstimate}\underset{\delta \geq 0}\sup \left|\xi_{\alpha}\partial^{(\alpha)}_{x}(x^\delta)\right|_{Y}= \underset{\delta \geq 0}\sup \left|\binom{\delta}{\alpha} \xi_\alpha x^{\delta - \alpha}\right|_{Y} \leq 1\end{equation}
where the last inequality follows from the facts that $\xi_\alpha $ is a monic polynomial in $K^\circ[x]$, and that $|x|_Y = 1$. This already implies that when $n = 0$ (and therefore, when $Y = X$), the operator norm of $\xi_\alpha \partial^{(\alpha)}_x$ on $Y$ is bounded above by $1$ for any $\alpha$. Since $|\pi^\alpha| \to 0$ as $\alpha \to \infty$, it follows that $(\xi_{\alpha} \pi^\alpha \partial^{(\alpha)}_{x|Y})_{\alpha=0}^{\infty}$ tends to zero in the Banach algebra $\underline{\End}_{K}(\O(Y))$ in this case.

Next, suppose that $n=1$, and write $Y:= \Sp K\langle x,\frac{\tau}{x-a}\rangle$. Then  \[\{x^{\delta} \ \ | \ \ \delta\geq 0\} \cup \{z_\beta := \left(\frac{\tau}{x-a}\right)^{\beta+1} \ \ | \ \  \beta \geq 0\}  \]
is a topological basis for $\mathcal{O}(Y)= K\langle x,\frac{\tau}{x-a}\rangle$, consisting of elements of supremum norm $1$. Because $|z_\beta|_Y = 1$ and $z_{\beta} \in \mathcal{O}(Y)^{\times}$, we see that
\[|\xi_\alpha \partial^{(\alpha)}_{x}(z_{\beta})|_{Y}=|z_{\beta} z^{-1}_{\beta}\xi_\alpha \partial^{(\alpha)}_{x}(z_{\beta})|_{Y}\leq |z^{-1}_{\beta}\xi_\alpha  \partial^{(\alpha)}_{x}(z_\beta)|_{Y}.
\]
Next, notice that 
\begin{equation}\label{LogDer}\begin{split}z^{-1}_{\beta} \partial^{(\alpha)}_{x}(z_\beta)& =(x-\alpha)^{\beta+1}\binom{\alpha+\beta}{\alpha}(-1)^{\alpha}(x-\alpha)^{-\beta-1-\alpha} \\ &=(-1)^{\alpha}\binom{\alpha+\beta}{\alpha}(x-a)^{-\alpha}.\end{split}\end{equation}
Let $y=x-a$ and let $\gamma$ be the unique non-negative integer such that $a\in \lambda_\gamma + \mathfrak{m}$. Define $\rho$ as $a-\lambda_\gamma$, so that $|\rho|<1$. Then
\[\left(\frac{(y+\rho)^{\alpha}}{y} \right)^{\alpha}= \left( \rho^\alpha y^{-1} + \binom{\alpha}{1}\rho^{\alpha-1}+ \binom{\alpha}{2}\rho^{\alpha-2}y+\cdots +y^{\alpha-1}\right)^{\alpha}
\]
which, recalling that $z_0=\frac{\tau}{x-a}=\frac{\tau}{y}$, we can rewrite as 
\begin{equation}\label{Yrho}\left(\frac{(y+\rho)^{\alpha}}{y} \right)^{\alpha}=(\frac{\rho^\alpha}{\tau}z_0+f_{\alpha}(x))^{\alpha} \quad \mbox{for some} \quad f_\alpha(x) \in K^\circ[x].\end{equation}
We have $|z_{\beta}|_Y=1$ and $|x - \lambda_\beta|_Y = 1$ for all $\beta \geq 0$. Hence, combining equations $(\ref{LogDer})$ and $(\ref{Yrho})$, we see that for any $\alpha \geq \gamma$ and any $\beta \geq 0$ we have
\[
\begin{split}|z_{\beta}^{-1}\xi_{\alpha}\partial_{x}^{(\alpha)}(z_{\beta})|_{Y}& =\left|\binom{\alpha+\beta}{\alpha}(x-\lambda_0)^{\alpha^2}\cdots (x-\lambda_\alpha)^{\alpha^2}(x-a)^{-\alpha}\right|_{Y} \\&\leq\left|\frac{(x-\lambda_\gamma)^{\alpha^2}}{(x-a)^{\alpha}}\right|_Y
\\&= \left|\left(\frac{(y+\rho)^{\alpha}}{y}\right)^{\alpha}\right|_{Y} \\ & = |\frac{\rho^{\alpha}}{\tau}z_0+f_{\alpha}(x)|^{\alpha}_{Y}.
\end{split}
\]
Since $|z_0|_Y = 1$ and $|f_\alpha(x)|_Y \leq 1$, we see that for any $\alpha \geq \gamma$ and any $\beta \geq 0$ we have
\[|z_{\beta}^{-1}\xi_{\alpha}\partial_{x}^{(\alpha)}(z_{\beta})|_{Y}\leq (\max \{\frac{|\rho|^\alpha}{|\tau|},1\})^{\alpha}.\]
However, since $|\rho| < 1$, $|\rho|^\alpha / |\tau|$ tends to zero as $\alpha \to \infty$. Hence
\[ C := \sup_{\alpha \geq 0} (\max \{\frac{|\rho|^\alpha}{|\tau|},1\})^{\alpha}\]
is finite, and we deduce that 
\begin{equation}\label{BetaEstimate}|\xi_{\alpha} \partial_{x}^{(\alpha)}(z_\beta)|_Y \leq C \quad\mbox{whenever} \quad \alpha \geq \gamma \quad\mbox{and}\quad \beta \geq 0.\end{equation}
Because $\{x^\delta, z_\beta : \delta,\beta\in \mathbb{N}\}$ forms a topological basis for $\O(Y)$ consisting of elements of supremum norm $1$, we combine estimates $(\ref{DeltaEstimate})$ and $(\ref{BetaEstimate})$ to obtain the bound
\[||(\xi_\alpha \partial^{(\alpha)}_x)|_{Y}|| \leq \sup \{|\xi_\alpha \partial^{(\alpha)}_x (x^{\delta})|_Y,|\xi_\alpha \partial^{(\alpha)}_x (z_\beta)|_Y : \beta,\delta\in \mathbb{N}\} \leq \max\{1,C\} = C
\]
whenever $\alpha \geq \gamma$. Since $|\pi^\alpha| \to 0$ as $\alpha \to \infty$, we see that $(\xi_{\alpha} \pi^\alpha \partial^{(\alpha)}_{x|Y})_{\alpha=0}^{\infty}$ tends to zero in $\underline{\End}_{K}(\O(Y))$, as claimed.

Finally, suppose that $n\geq 2$. Recall that 
\[Y = X-\left(B(a_1, |\tau_1|) \cup \cdots \cup B(a_n, |\tau_n|) \right),\] 
where the open discs $B(a_i,|\tau_i|)$ are pairwise disjoint for $i=1,\ldots n$. Let 
\[T := X - B(a_n, |\tau_n|)\quad\mbox{and} \quad Z:=X-\left(B(a_1, |\tau_1|) \cup \cdots \cup B(a_{n-1}, |\tau_{n-1}|) \right).\]
Because these affinoid subdomains of $X$ satisfy
\[Z \cup T= X\quad\mbox{and}\quad Z \cap T  = Y,\]
we have the exact sequence of $\mathcal{D}(X)$-modules 
\[0\to \O(X) \to \O(Z)\oplus \O(T) \to \O(Y)\to 0\]
by Tate's Acyclicity Theorem \cite[Theorem 4.2.2]{FvdPut}. The operators $\xi_\alpha\pi^\alpha \partial^{(\alpha)}_x$ converge to zero as $\alpha \to \infty$ with respect to the operator norm on $\O(X)$, $\O(Z)$ and $\O(T)$ by induction. Therefore they also do so on $\O(Y)$. \qed

\begin{proof}[Proof of Claim \ref{XiNotInDCap}] By Lemma \ref{lem:symbol}, it is enough to show that the family $(\xi_{\alpha} \pi^{\alpha} / \alpha!)_{\alpha \in \mathbb{N}}$ is \emph{not} rapidly decreasing.  Now $|x - \lambda|_X = 1$ for any $\lambda \in K^\circ$ because $|\cdot|_X$ is the Gauss norm on $K\langle x \rangle$, so
\[|\xi_{\alpha}|_X = |(x-\lambda_{0})^{\alpha^{2}}\cdots (x-\lambda_{\alpha})^{\alpha^{2}}|_X = 1, \qmb{for all} \alpha \geq 0.\]
Now take $n = 2$. Since $|\alpha!| \leq 1$ for all $\alpha \geq 0$, we obtain
\[ \left\vert \frac{ \xi_\alpha \pi^{\alpha} }{\alpha! \pi^{2 \alpha} } \right\vert \geq |\pi|^{-\alpha} \qmb{for all} \alpha \in \mathbb{N}.\]
Since $|\pi|^{-1} >1$, this tends to $+\infty$ as $\alpha \to \infty$.
\end{proof}

\begin{prop}\label{atob} (a) $\Rightarrow$ (b) as in Theorem \ref{Main}.
\end{prop}
We now finish the second part of the proof of Theorem \ref{Main}, having proven the first part in Proposition \ref{btoa}.
\begin{proof} [Proof of Theorem \ref{Main} (a) $\Rightarrow$ (b)]
Claims \ref{Claim1} and \ref{XiNotInDCap} show that the residue field of $K$ has to be uncountable, otherwise $\rho_X$ is not surjective when $X = \Sp K \langle x \rangle$ is the closed unit disc. On the other hand, let $L$ be a finite field extension of $K$, and let $X = \Sp L$, a smooth rigid $K$-analytic variety of dimension zero. Then $\w\D(X) = L$ and $\E(X) = \End_K(L)$, and the map $\rho(X) : L \to \End_K(L)$ is the natural inclusion given by the action of $L$ on itself by left-multiplicatin. If $\rho(X)$ is an isomorphism then this forces $\dim_KL = 1$, so $L = K$ and $K$ must be algebraically closed.
\end{proof}
\begin{lem}The maps $\rho$ and $\eta$ take bounded sets to bounded sets. Therefore, when $X$ is a smooth rigid analytic variety over a non-trivially valued, non-Archimedean field which is algebraically closed and has uncountable residue field, we have established an isomorphism of bornological (and hence Fr\'{e}chet) sheaves $\E \cong \w\D$ over $X_w$.
\end{lem}
\begin{proof}
In order to show that $\eta$ preserves bounded sets, we should prove that for any $U \in X_w$ and $B$ bounded in $\E(U)$ that $\eta(U)(B)$ is bounded in $\w\D(U)$. Choose an admissible cover of $U$ by affinoids $U_i$ together with \'{e}tale morphisms $g_i:U_i \to \mathbb{D}^{n_i}$.  It is enough to show that $\{\eta(U)(\varphi)|_{U_i}\}$ is bounded in $\w\D(U_i)$. We need to show that for each $i$ and real number $r>0$, there is a constant $C=C(i,r)$ such that 
\[\underset{\alpha\in \mathbb{N}^{n_i}}\sup |\eta_{\alpha}(U_i)(\varphi)|r^{\alpha} \leq C
\]
for all $\varphi \in B$. Now $\{ \varphi(U_i) |\varphi \in B \}$ is a bounded subset of $\underline{\End}_{K}(\O(U_i))$ and so because the map 
\[\underline{\End}_{K}(\O(U_i)) \to \mathbb{R}_{\geq 0}
\]
defined by 
\[\varphi \mapsto \underset{\alpha\in \mathbb{N}^{n_i}}\sup |\eta_{\alpha}(U_i)(\varphi)|_{U_i}r^{\alpha} =\underset{\alpha\in \mathbb{N}^{n_i}}\sup \left|\sum_{\beta\leq\alpha}\binom{\alpha}{ \beta }(\varphi(U_i)(x^{\beta}))(-x)^{\alpha-\beta}\right|_{U_i}r^{\alpha} 
\]
is bounded (this map is well defined by Proposition \ref{btoa}), the image of  $\{ \varphi(U_i) |\varphi \in B \}$ is bounded and so we can find $C$ as needed.

In order to show that $\rho$ preserves bounded sets, it is enough according to Equation (\ref{inthom}) to prove that any $U \in X_w$ and for any bounded set $B \subset \w\D(U)$ that the projection $\rho(U,V)(B)$ of $\rho(U)(B)$ to $\underline{\End}_{K}(\O(V))$ is bounded in the operator norm for any $V\in U_w$. Now, \cite[Lemma 7.6(a)]{DCapOne} gives us an $\O(U)^\circ$-Lie lattice $\L$ in $\T(U)$ such that $\L \cdot \O(V)^\circ \subseteq \O(V)^\circ$. Expressed in a different way, $\rho(U,V)(\L)$ is contained in the unit ball, $\C$ say, of the Banach algebra $\underline{\End}_{K}(\O(V))$. It follows that $\rho(U,V)( \h{U(\L)} ) \subseteq \C$, also. The map $\w\D(U) \to \hK{U(\L)}$ is bounded by the definition of the bornology on the Fr\'echet algebra $\w\D(U)$, so we can find a scalar $\lambda \in K$ such that $B \subseteq \lambda \h{U(\L)}$. Therefore
\[ \rho(U,V)(B) \subseteq \rho(U,V)( \lambda \h{U(\L)} ) \subseteq \lambda \C\]
and hence $\rho(U,V)(B)$ is a bounded subset of $\underline{\End}_{K}(\O(V))$ as required.\end{proof}
\section{Conclusions and outlook} In future work we plan to investigate non-Archimedean contexts other than the rigid analytic one for which it may be possible to show the equivalence of the two sheaves in greater generality. We believe there is an interesting statement to be made for non-Archimedean fields that are non-necessarily algebraically closed or with uncountable residue field. We also intend to study the Riemann-Hilbert correspondence in the non-Archimedean context, hoping to view it as an instance of derived Morita equivalence. This requires proving that the higher sheaf-Ext groups of the structure sheaf with itself vanish. This was the point of view taken in the context of complex manifolds in \cite{Meb} and \cite{PrSch} based in part on ideas of Kashiwara.
\bibliographystyle{plain}
\bibliography{references}
\end{document}